\documentclass[12pt]{amsart}

\usepackage{mathtext}
\usepackage[cp1251]{inputenc}
\usepackage[T2A]{fontenc}
\usepackage[ukrainian,english]{babel}

\usepackage[dvips]{graphicx}
\usepackage{amsmath}
\usepackage{amssymb}
\usepackage{amsxtra}
\usepackage{latexsym}
\usepackage{ifthen}
\usepackage[centerlast,footnotesize,sl]{caption2}

\newtheorem{theorem}{Theorem}
\newtheorem{lemma}{Lemma}
\renewenvironment {proof} {\begin{trivlist} \item[\hspace{\labelsep}%
\emph {Proof.}]}{$\Box$ \end{trivlist}}

\begin{document}
\title{Analog of the mean-value theorem for polynomials of special form}
\author{O.D. Trofymenko}

\address{Faculty of Mathematics and Information Technology, Donetsk National University, Universitetskaya str.~24, 83001, Donetsk,  Ukraine}
 \email{ol4anovskiy@gmail.com}

\maketitle

\begin{abstract}
 In the present paper a new mean value theorem for polynomials of
special form is obtained. The case of sums on vertices of a regular
polygon is studied. A criterion  for a certain equation to be
satisfied is obtained.
\end{abstract}

\section{Introduction}
   The classical Gaussian theorem characterizes the class of harmonic functions by the mean value formula.
 This theorem has received further development and elaboration in many papers (see, for example, reviews ~\cite{V1}, ~\cite{V2}
 and monographs ~\cite{V3}, ~\cite{V4} with extensive bibliography).
 One of the main ways in
 this  study is a description for classes of functions. This classes satisfy given integral equations,
 that have certain geometric meaning. There are mean value theorems that characterize harmonic
 polynomials ~\cite{V5}, bianalytic functions ~\cite{V6}, a solution
 of convolution equations with finite convolver and others (see ~\cite{V7}).
 In addition, similar results are very important in integral geometry
 and  various applications (see ~\cite{V3}).

   In this paper the mean value theorem for polyanalytic polynomials
   of a special form is obtained. A feature of it is function's
   value on vertices of a regular polygon in the mean value formula (see below).
   Also there is a value of some differential operator in the center
   of this polygon.

    We need the following notation for accurate formulation of the
    main result.

   Assume that $s,h\in{\mathbb{Z}}$,
$n,m\in{\mathbb{N}}$,  $n\geq{3}$,
 $0\leq{h}<n-s$,  $0\leq{s}\leq{m-1}$,  $q=\min\{h+s,m-1\}$.

Let $B_{\mathcal{R}}$ be the disk in $\mathbb{C}$ centered at origin
with radius $\mathcal{R}$. Denote by
$\zeta_\nu=Re^{i\frac{2\pi\nu}{n}}$ ($\nu=1,...,n$) the vertices of
a regular $n$-gon with circumradius $R$ and inscribed radius $r$.

Let  $\mathcal{R}_\ast(n,R,r)=
  \begin{cases}
    \sqrt{5R^2+4rR} & \text{where n is odd}, \\
     \sqrt{8R^2+R^4/r^2} & \text{where n is even}.

  \end{cases}
$

\section{Formulation of the Main Result}

 \begin{theorem}\label{t1}

Let $\mathcal{R}>\frac{1}{2}\mathcal{R}_\ast(n,R,r)$ and
$f\in{C^q}(B_{\mathcal{R}})$. Then the next statements are
equivalent.

1) For all $0\leq\alpha\leq2\pi$, $z+\zeta_\nu
e^{i\alpha}\in{B_{\mathcal{R}}}$, $\nu\in{\{1,...,n\}}$ there holds
\begin{equation}\label{eq1}
 \overset{n}{\underset{\nu=1}{\sum}}{(\zeta_\nu e^{i\alpha})}^sf(\zeta_\nu
e^{i\alpha}+z)=\overset{q}{\underset{p=s}{\sum}}\frac{nR^{2p}}{(p-s)!p!}{\left(\frac{\partial}{\partial{z}}\right)}^{p-s}{\left(\frac{\partial}{\partial{\bar{z}}}\right)}^pf(z).
\end{equation}

2) The function $f$ has the following form
\begin{equation}\label{eq2}
f(z)=\overset{h}{\underset{k=0}{\sum}}{\overset{m-1}{\underset{l=0}{\sum}}}c_{k,l}z^k\bar{z}^l,
\end{equation}
where $c_{k,l}$ are some constants.
\end{theorem}

\section{Auxiliary Statements}
\begin{lemma}\label{l1} Let a function $f(z)$ have the form (\ref{eq2}).

 Then the following equality holds
\begin{equation}\label{eq3}
\overset{n}{\underset{\nu=1}{\sum}}{(\zeta_\nu e^{i\alpha})}^sf(\zeta_\nu
e^{i(\alpha+\beta)}+z
e^{i\beta})=\overset{q}{\underset{p=s}{\sum}}\frac{nR^{2p}}{(p-s)!p!}{\left(\frac{\partial}{\partial{z}}\right)}^{p-s}{\left(\frac{\partial}{\partial{\bar{z}}}\right)}^pf(z
e^{i\beta}),
\end{equation}
where  $0\leq\alpha\leq2\pi$, $0\leq\beta\leq2\pi$,
$ze^{i\beta}+\zeta_\nu e^{i(\alpha+\beta)}\in{B_{\mathcal{R}}}$.
\end{lemma}
\begin{proof}
We have the following
\[
e^{i\alpha
s}\overset{n}{\underset{\nu=1}{\sum}}\zeta_\nu^sf(\zeta_\nu
e^{i(\alpha+\beta)}+z
e^{i\beta})=\overset{n}{\underset{\nu=1}{\sum}}\overset{h}{\underset{k=0}{\sum}}{\overset{m-1}{\underset{l=0}{\sum}}}\overset{k}{\underset{j=0}{\sum}}{\overset{l}{\underset{p=0}{\sum}}}c_{k,l}C_k^jC_l^pe^{i\alpha
s}\zeta_\nu^{s+j}\bar{\zeta_\nu}^p e^{i\alpha (j-p)} e^{i\beta k}e^{-i\beta
l}z^{k-j}\bar{z}^{l-p}.
\]
Separating the concerns of the $z$ and  $\zeta_\nu$, we get
\[
e^{i\alpha
s}\overset{h}{\underset{j=0}{\sum}}{\overset{m-1}{\underset{p=0}{\sum}}}e^{i\alpha (j-p)}\frac{1}{j!p!}{\left(\frac{\partial}{\partial{z}}\right)}^{j}{\left(\frac{\partial}{\partial{\bar{z}}}\right)}^pf(ze^{i\beta})\overset{n}{\underset{\nu=1}{\sum}}R^{s+j+p}e^{i(s+j-p)\frac{2\pi\nu}{n}}=0,
\]
except the case, where  $s+j-p=q_1n$, $q_1\in\mathbb{Z}$ .

Let us estimate $q_1n$:
\[
1-m\leq{q_1n}\leq{m-1+h}\leq{m-1+n-s}\leq{m-1+n}.
\]
This yields: $q_1n=0$.

 So, $s+j-p=0$.
Then
\[
e^{i\alpha
s}\overset{h}{\underset{j=0}{\sum}}{\overset{m-1}{\underset{p=0}{\sum}}}e^{i\alpha (j-p)}\frac{1}{j!p!}{\left(\frac{\partial}{\partial{z}}\right)}^{j}{\left(\frac{\partial}{\partial{\bar{z}}}\right)}^pf(ze^{i\beta})\overset{n}{\underset{\nu=1}{\sum}}R^{s+j+p}e^{i(s+j-p)\frac{2\pi\nu}{n}}=
\]
\[
=\overset{q}{\underset{p=s}{\sum}}\frac{nR^{2p}}{(p-s)!p!}{\left(\frac{\partial}{\partial{z}}\right)}^{p-s}{\left(\frac{\partial}{\partial{\bar{z}}}\right)}^pf(z
e^{i\beta}).
\]

Then it is obvious that equality (\ref{eq3}) holds.
\end{proof}
Let us  introduce the corresponding Fourier series for a function
$f\in{C^q}(B_{\mathcal{R}})$
\begin{equation}\label{eq4}
f(z)=\overset{\infty}{\underset{k=-\infty}{\sum}}f_{k}(\rho)e^{i\varphi
k},
\end{equation}
where
$f_{k}(\rho)=\frac{1}{2\pi}\overset{\pi}{\underset{-\pi}{\int}}f(\rho
e^{i\varphi})e^{-i\varphi k}d\varphi$.
\begin{lemma}\label{l2}
Assume that $f\in{C^q}(B_{\mathcal{R}})$
 and that for this function equality (\ref{eq3}) holds. Then  this
 equality holds for every term of the Fourier series of this function.
 The converse is also true.
\end{lemma}
\begin{proof}
$\textbf{Necessity}$.

 We multiply the left and the right sides of
(\ref{eq3}) by $e^{-i\beta k}$ and integrate over $\beta$ from
$-\pi$ to $\pi$. We have
\[
\overset{\pi}{\underset{-\pi}{\int}}e^{-i\beta k}\overset{n}{\underset{\nu=1}{\sum}}{(\zeta_\nu e^{i\alpha})}^s\overset{\infty}{\underset{k=-\infty}{\sum}}f_{k}(\rho')e^{i(\varphi-\frac{2\pi\nu}{n}-\alpha)
k}e^{i\beta k}d\beta=
\]
\[
=\overset{\pi}{\underset{-\pi}{\int}}e^{-i\beta k}\overset{q}{\underset{p=s}{\sum}}\frac{nR^{2p}}{(p-s)!p!}{\left(\frac{\partial}{\partial{z}}\right)}^{p-s}{\left(\frac{\partial}{\partial{\bar{z}}}\right)}^p\overset{\infty}{\underset{k=-\infty}{\sum}}f_{k}(\rho)e^{i\varphi k}e^{i\beta k}d\beta,
\]
where $\rho'=\sqrt{R^2+\rho^2+2R\rho
\cos(\varphi-\frac{2\pi\nu}{n}-\alpha)}$.

Next,
\[
\overset{n}{\underset{\nu=1}{\sum}}{(\zeta_\nu e^{i\alpha})}^sf_{k}(\rho')e^{i(\varphi-\frac{2\pi\nu}{n}-\alpha)
k}=\overset{q}{\underset{p=s}{\sum}}\frac{nR^{2p}}{(p-s)!p!}{\left(\frac{\partial}{\partial{z}}\right)}^{p-s}{\left(\frac{\partial}{\partial{\bar{z}}}\right)}^pf_{k}(\rho)e^{i\varphi
k},
\]
which concludes the proof.

$\textbf{Sufficiency}$.

Let
\[
\lambda(\alpha)=\overset{n}{\underset{\nu=1}{\sum}}{(\zeta_\nu e^{i\alpha})}^sf(\rho'\cos(\varphi-\frac{2\pi\nu}{n}-\alpha+\beta),\rho'\sin(\varphi-\frac{2\pi\nu}{n}-\alpha+\beta))-
\]
\[
-\overset{q}{\underset{p=s}{\sum}}\frac{nR^{2p}}{(p-s)!p!}{\left(\frac{\partial}{\partial{z}}\right)}^{p-s}{\left(\frac{\partial}{\partial{\bar{z}}}\right)}^pf(\rho
\cos(\varphi+\beta),\rho \sin(\varphi+\beta)).
\]
Then we have the next equality
\[
\overset{\pi}{\underset{-\pi}{\int}}\lambda(\rho)e^{-i\beta k}d\rho=\overset{n}{\underset{\nu=1}{\sum}}\left(\overset{\pi}{\underset{-\pi}{\int}}{(\zeta_\nu e^{i\alpha})}^sf(\rho'\cos\beta,\rho'\sin\beta)e^{-i\beta k}d\beta\right)e^{i(\varphi-\frac{2\pi\nu}{n}-\alpha)
k}-
\]
\[
-\overset{q}{\underset{p=s}{\sum}}\frac{nR^{2p}}{(p-s)!p!}{\left(\frac{\partial}{\partial{z}}\right)}^{p-s}{\left(\frac{\partial}{\partial{\bar{z}}}\right)}^p\overset{\pi}{\underset{-\pi}{\int}}f(\rho
\cos\beta,\rho \sin\beta)e^{-i\beta k}d\beta
e^{i k\varphi}=0.
\]
Thus, equality $\lambda(\alpha)=0$ completes the proof of Lemma.
\end{proof}

\begin{lemma}\label{l3}
Let $f(z)=cN_0(\lambda|z|)$, where $N_0(\lambda|z|)$ is the Neumann
function, $\lambda\neq0$, $c$ is a constant, and  $f(z)$ satisfies
(\ref{eq1}) in $B_{\mathcal{R}}$. Then $c=0$.
\end{lemma}
\begin{proof}
The function $N_0(\lambda|z|)$ is real-analytic, so from [3, Part 1]
we have
\[
N_0(\lambda|z|)=\frac{2}{\pi}J_0(\lambda|z|)\left(\log\frac{\lambda|z|}{2}+\gamma\right)-\frac{2}{\pi}\overset{\infty}{\underset{m=0}{\sum}}\frac{(-1)^m{\left(\frac{\lambda|z|}{2}\right)}^{2m}}{(m!)^2}\overset{m}{\underset{k=1}{\sum}}\frac{1}{k} ,
\]
where $J_0(\lambda|z|)$ is the Bessel function,
$\gamma=\underset{N\rightarrow
+\infty}{\lim}\left(\overset{N}{\underset{k=1}{\sum}}\frac{1}{k}-\log
N\right)$. Using the Taylor expansion of the Bessel functions, we
obtain
\[
N_0(\lambda|z|)=\frac{2}{\pi}\overset{\infty}{\underset{m=0}{\sum}}\frac{(-1)^m{(\frac{\lambda|z|}{2})}^{2m}}{(m!)^2}\left(\log\frac{\lambda|z|}{2}+\gamma-\overset{m}{\underset{k=1}{\sum}}\frac{1}{k}\right).
\]
 Let us substitute  $f(z)=cN_0(\lambda|z|)$ into equality
 (\ref{eq1}).

We have
\[
\overset{n}{\underset{\nu=1}{\sum}}{(\zeta_\nu e^{i\alpha})}^s
cN_0(\lambda|\zeta_\nu e^{i\alpha}+z|)=\overset{q}{\underset{p=s}{\sum}}c\frac{nR^{2p}}{(p-s)!p!}{\left(\frac{\partial}{\partial{z}}\right)}^{p-s}{\left(\frac{\partial}{\partial{\bar{z}}}\right)}^p
N_0(\lambda|z|),
\]
where the function  ${(\zeta_\nu e^{i\alpha})}^s
N_0(\lambda|\zeta_\nu e^{i\alpha}+z|)$ is real-analytic.

Denote the derivation of the Neumann function in the last equation
by $DN_0(\lambda|z|)$, which is a real-analytic function. According
to [3, Part 1, Prop.7.1] we have $D=(\Delta+\lambda^2)P(\partial)$,
where $P(\partial)$ is a differential operator in $\mathbb{R}^2$.
Then $DN_0(\lambda|z|)=0$ everywhere except  $z=0$.

On the other hand, using the obtained form for $N_0(\lambda|z|)$, we
have
\[
\overset{n}{\underset{\nu=1}{\sum}}{(\zeta_\nu e^{i\alpha})}^s
cN_0(\lambda|\zeta_\nu e^{i\alpha}+z|)\equiv0
\]
in a neighborhood of the point $z=0$.

 Hence, the constant $c$ is zero in the general form of the function $f(z)$.
\end{proof}
\begin{lemma}
Let $\gamma\in\mathbb{R}^1$, $z=x+iy$,
$\lambda\in\mathbb{C}\setminus{\{0\}}$, $c\in\mathbb{C}$ and
$f^{\ast}(z)=ce^{i(x\cos\gamma+y\sin\gamma)\lambda}$. Then
$f^{\ast}(z)=\overset{\infty}{\underset{k=-\infty}{\sum}}c_kJ_k(\lambda\rho)e^{ik\varphi}$,
where $c_k$ are constants.
\end{lemma}
\begin{proof}
 At first let us show that the initial function satisfies the following equation
\begin{equation}\label{eq5}
\Delta{f^{\ast}(z)}+\lambda^2f^{\ast}(z)=0.
\end{equation}
We have
\[
-c\lambda^2e^{i(x\cos\gamma+y\sin\gamma)\lambda}+\lambda^2ce^{i(x\cos\gamma+y\sin\gamma)\lambda}=0.
\]
 Now we check that each term  $f_k(\rho)e^{ik\varphi}$ of the expansion of $f^{\ast}(z)$  satisfies
equation (\ref{eq5}) too.

Let
\begin{equation}\label{eq6}
f_k(\rho)e^{ik\varphi}=\frac{1}{2\pi}\overset{\pi}{\underset{-\pi}{\int}}f^{\ast}(x\cos
t-y\sin t,x\sin t+y\cos t)e^{-ikt}dt.
\end{equation}
We denote $h(x,y,t)=f^{\ast}(x\cos t-y\sin t,x\sin t+y\cos t)$.

Then
\[
\frac{1}{2\pi}\overset{\pi}{\underset{-\pi}{\int}}\Delta{h(x,y,t)}e^{-ikt}dt+\lambda^2\frac{1}{2\pi}\overset{\pi}{\underset{-\pi}{\int}}h(x,y,t)e^{-ikt}dt=0.
\]
Now it is clear that
\[
\Delta{(f_k(\rho)e^{ik\varphi})}+\lambda^2f_k(\rho)e^{ik\varphi}=0
\]
It is known that
\[\frac{\partial}{\partial
x}\left(f_k(\rho)e^{ik\varphi}\right)=\frac{1}{2}\left(f^{'}_k+\frac{k}{\rho}f_k\right)e^{i(k-1)\varphi}+\frac{1}{2}\left(f^{'}_k-\frac{k}{\rho}f_k\right)e^{i(k+1)\varphi}\]
and
\[\frac{\partial}{\partial
y}\left(f_k(\rho)e^{ik\varphi}\right)=\frac{i}{2}\left(f^{'}_k+\frac{k}{\rho}f_k\right)e^{i(k-1)\varphi}-\frac{i}{2}\left(f^{'}_k-\frac{k}{\rho}f_k\right)e^{i(k+1)\varphi}.\]

Now we can find the following
\[
\frac{\partial^2}{\partial
x^2}\left(f_k(\rho)e^{ik\varphi}\right)=\frac{1}{4}{\left(\left(f^{'}_k+\frac{k}{\rho}f_k\right)'+\frac{k-1}{\rho}\left(f^{'}_k+\frac{k}{\rho}f_k\right)\right)}e^{i(k-2)\varphi}+\frac{1}{4}{\left(2f''_k+2\frac{f'_k}{\rho}-2\frac{k^2}{\rho^2}f_k\right)}e^{ik\varphi}+
\]
\[
+\frac{1}{4}{\left(\left(f^{'}_k-\frac{k}{\rho}f_k\right)'-\frac{k+1}{\rho}\left(f^{'}_k-\frac{k}{\rho}f_k\right)\right)}e^{i(k+2)\varphi}.
\]
Similarly,
\[
\frac{\partial^2}{\partial
y^2}\left(f_k(\rho)e^{ik\varphi}\right)=-\frac{1}{4}{\left(\left(f^{'}_k+\frac{k}{\rho}f_k\right)'+\frac{k-1}{\rho}\left(f^{'}_k+\frac{k}{\rho}f_k\right)\right)}e^{i(k-2)\varphi}+\frac{1}{4}{\left(2f''_k+2\frac{f'_k}{\rho}-2\frac{k^2}{\rho^2}f_k\right)}e^{ik\varphi}-
\]
\[
-\frac{1}{4}{\left(\left(f^{'}_k-\frac{k}{\rho}f_k\right)'-\frac{k+1}{\rho}\left(f^{'}_k-\frac{k}{\rho}f_k\right)\right)}e^{i(k+2)\varphi}.
\]
Considering the above, we have the Bessel equation for the function
$f_k(\rho)$
\begin{equation}\label{eq7}
\rho^2f''_k+\rho f'_k+(\lambda^2\rho^2-k^2)f_k=0.
\end{equation}
Since $f\in{C^\infty(\mathbb{C})}$, it follows from (\ref{eq6}) that
the function $f_k(\rho)$ is continuous on $[0,+\infty]$. Hence from
equality (\ref{eq7}) we obtain $f_k(\rho)=c_kJ_k(\lambda\rho)$.
\end{proof}
\begin{lemma}
Let $f(z)=cJ_0(\lambda|z|)$ ($\lambda\neq0$, $c\in\mathbb{C}$)
satisfy (\ref{eq1}) in $B_{\mathcal{R}}$. Then $c=0$.
\end{lemma}
\begin{proof}
From Lemma 4, $f(z)=ce^{i(x\cos\gamma+y\sin\gamma)\lambda}$.

 Substitute $f(z)$ in equation (\ref{eq1}).
\[
\overset{n}{\underset{\nu=1}{\sum}}{(\zeta_\nu e^{i\alpha})}^sce^{i(x_\nu
\cos\gamma+y_\nu
\sin\gamma)\lambda}e^{i(x\cos\gamma+y\sin\gamma)\lambda}=\overset{q}{\underset{p=s}{\sum}}\frac{nR^{2p}}{(p-s)!p!}c\frac{i^{2p-s}}{2^{2p-s}}\lambda^{2p-s}e^{i\gamma
s}e^{i(x\cos\gamma+y\sin\gamma)\lambda}.
\]
It is obvious that $c=0$.
\end{proof}

Denote by $C_\natural^q$ a class of $q$ times differentiated radial
functions.
\begin{lemma}\label{l4}
Let  $f\in{C_\natural^q}(B_{\mathcal{R}})$ and assume that
(\ref{eq1}) holds in $B_{\mathcal{R}}$. Then
$f(z)=\overset{q}{\underset{k=0}{\sum}}c_k|z|^{2k}$, where $c_k$ are
some constants.
\end{lemma}
\begin{proof}
From [3, Part 4, Th.3.2] we have the next statement.

Let $f\in{C^q(B_{\mathcal{R}})}$ and assume that there exists a
polynomial $Q$:
 $\mathbb{C}\rightarrow\mathbb{C}$ satisfying the equality
\[
\overset{n}{\underset{\nu=1}{\sum}}{(\zeta_\nu e^{i\alpha})}^sf(\zeta_\nu
e^{i\alpha}+z)=\overset{q}{\underset{p=s}{\sum}}\frac{nR^{2p}}{(p-s)!p!}Q(\partial)f(z).
\]

 Then there exists a polynomial $P$ such that
$P(\triangle)f_0=0$ in $B_{\mathcal{R}}$, where $\triangle$ is the
Laplacian.

Therefore,
$P(\triangle)f=\overset{n}{\underset{\nu=1}{\sum}}c_\nu\triangle^\nu
f_0=0$, where $c_\nu$ are constants.

Thus for any $i\in{{1,...,n}}$ we have the following
\[
 (\triangle-\lambda_i)F=0,
 \]
where $F=c_n \overset{n}{\underset{j\neq
i}{\prod}}(\triangle-\lambda_j)f_0$,  $\lambda_i\neq0$ are solutions
of $P(z)=0$.

 Let $\lambda_i\neq0$.

We find $F(z)$ as the solution of the differential Bessel equation
\[
F(z)=c_1J_0(\sqrt{\lambda_i}|z|)+c_2N_0(\sqrt{\lambda_i}|z|),
\]
where $J_0, N_0$ are the Bessel and Neumann functions respectively,
and $c_1, c_2$ are some constants.

From Lemma 3 and Lemma 5 we see that $c_1=c_2=0$.

Hence $\overset{n}{\underset{j\neq
i}{\prod}}(\triangle-\lambda_j)f_0=0$.

 Then
$(\triangle-\lambda_m)\overset{n}{\underset{j\neq
m}{\prod}}(\triangle-\lambda_j)f_0=0$, $m\in\mathbb{N}$.

Similarly we obtain, that $\triangle f_0=0$, $\lambda_j\neq0$.

If $\lambda_j=0$, then let
$f=\overset{N}{\underset{k>\min\{h,m-1\}}{\sum}}c_k|z|^{2k}$,
$N\in\mathbb{N}$.

Consider the case $k=\min\{h,m-1\}+1$.

Function with the selected index does not satisfy equation
(\ref{eq1}). Now take the Laplacian of $f$. As a result, we have
\[
\triangle\overset{N}{\underset{k>\min\{h,m-1\}}{\sum}}c_k|z|^{2k}=\overset{N}{\underset{k>\min\{h,m-1\}}{\sum}}\tilde{c_k}|z|^{2(k-1)},
\]
where $\tilde{c_k}$ are some constants.

Now substitute the function in equation (\ref{eq1}). The only case,
when $k=\min\{h,m-1\}+1$ fits for the initial equality. This
means that $f=\tilde{c}_{\min\{h,m-1\}}|z|^{2(k-1)}$, where
$\tilde{c}_{\min\{h,m-1\}}$ are constants.

Again, we take the Laplacian and  obtain the following function
\[
f=\overset{q}{\underset{k=0}{\sum}}c_k|z|^{2k},\ c_k = const,
\]
as contended.
\end{proof}
\begin{lemma}\label{l5}
Let  $f_{j}(\rho)e^{i j\varphi}$ satisfy (\ref{eq1}). Then  we have
the same equality
 for $f_{j+1}(\rho)e^{i (j+1)\varphi}$ and $f_{j-1}(\rho)e^{i (j-1)\varphi}$,
$j=0,1,2,...$ .
\end{lemma}
\begin{proof}
We obtaine for $f_{j}(\rho)e^{i j\varphi}$ the following equality
\[
\frac{\partial}{\partial{x}}\left(f_{j}(\rho)e^{i
j\varphi}\right)=\left(f_{j}'+j\frac{f_{j}(\rho)}{\rho}\right)e^{i(j-1)\varphi}+\left(f_{j}'-j\frac{f_{j}(\rho)}{\rho}\right)e^{i(j+1)\varphi},
\]
where
\begin{equation}\label{eq8}
 f_{j}'+j\frac{f_{j}(\rho)}{\rho}=f_{j-1}(\rho).
\end{equation}
 The function $f_{j}(\rho)e^{i j\varphi}$ satisfies (\ref{eq1}) by the
condition. Hence $\frac{\partial}{\partial{x}}\left(f_{j}(\rho)e^{i
j\varphi}\right)$ and $f_{j-1}(\rho)e^{i (j-1)\varphi}$ satisfy
(\ref{eq1}) too.

Similarly,
\begin{equation}\label{eq9}
 f_{j}'-j\frac{f_{j}(\rho)}{\rho}=f_{j+1}(\rho),
\end{equation}
and $f_{j+1}(\rho)e^{i (j+1)\varphi}$ satisfies (\ref{eq1}).
\end{proof}

\section{Proof of the Main Result.}

\begin{proof}
\textbf{Sufficiency}.

It is obviously that the equation (\ref{eq1}) and (\ref{eq3})
coincide for $\beta=0$.

Now Lemma 1 implies the sufficiency  for the main theorem.

\textbf{Necessity}.

We have from equality (\ref{eq8})
\[
f_1'(\rho)+\frac{f_1(\rho)}{\rho}=\overset{q}{\underset{k=0}{\sum}}c_k\rho^{2k}.
\]
Then
\[
f_1(\rho)=\overset{q}{\underset{k=0}{\sum}}c_k\frac{\rho^{2k+1}}{2k+2}+\frac{c}{\rho},
\]
where c is a constant.

 Substitute this function to equation (\ref{eq1}),
assuming that $c_{k,1}=\frac{c_k}{2k+2}$. We find
\[
\overset{n}{\underset{\nu=1}{\sum}}{(\zeta_\nu
e^{i\alpha})}^s\left(\overset{q}{\underset{k=0}{\sum}}c_{k,1}(\zeta_\nu
e^{i\alpha}+z)^{k+1}(\overline{\zeta_\nu}
e^{-i\alpha}+\overline{z})^k+\frac{c}{(\overline{\zeta_\nu}
e^{-i\alpha}+\overline{z})}\right)=
\]
\[
=\overset{q}{\underset{p=s}{\sum}}\frac{nr^{2p}}{(p-s)!p!}{\left(\frac{\partial}{\partial{z}}\right)}^{p-s}{\left(\frac{\partial}{\partial{\bar{z}}}\right)}^p
g(z)+\overset{n}{\underset{\nu=1}{\sum}}\frac{c}{(\overline{\zeta_\nu}
e^{-i\alpha}+\overline{z})},
\]
where $g(z)$ has the form (\ref{eq2}).  Hence it is clear that
equality holds provided that $c=0$ only. So,
$f_1(\rho)=\overset{q}{\underset{k=0}{\sum}}c_{k,1}\rho^{2k+1}$.

 Let
 $f_{j}(\rho)=\overset{q}{\underset{k=0}{\sum}}c_{k,j}\rho^{2k+j}$.
 Then
\[
f_{j+1}'(\rho)+(j+1)\frac{f_{j+1}(\rho)}{\rho}=\overset{q}{\underset{k=0}{\sum}}c_j\rho^{2k+j},
\]
and
$f_{j+1}(\rho)=\overset{q}{\underset{k=0}{\sum}}c_{k,j+1}\rho^{2k+j+1}$.

By induction we have
$f_j(\rho)=\overset{q}{\underset{k=0}{\sum}}c_{k,j}\rho^{2k+j}$.

 Now let us consider  functions
with negative indices. We start with $f_{-1}(\rho)$. From equality
(\ref{eq9}) we have the following
\[
f_{-1}'(\rho)+(-1)\frac{f_{-1}(\rho)}{\rho}=f_0(\rho).
\]
Hence, in the same way we get
\[
f_{-1}(\rho)=\overset{q}{\underset{k=0}{\sum}}c_{k,-1}\rho^{2k+1}.
\]
By induction we have again
\[
f_{-j}(\rho)=\overset{q}{\underset{k=0}{\sum}}c_{k,-j}\rho^{2k+j},
j\in\mathbb{N}.
\]
 This implies that
$f_{-j}(z)=(\overset{q}{\underset{k=0}{\sum}}c_{k,-j}\rho^{2k+j})e^{-i\varphi j}$.
We should carefully consider the following two cases.

1. Let $h\leq{m-1}$.
  Then
\[
 f_{-j}(z)=\overset{h}{\underset{k=0}{\sum}}c_{k,-j}z^{k}\bar{z}^{k+j}.
 \]

 Hence
\[
  \overset{n}{\underset{\nu=1}{\sum}}{(\zeta_\nu e^{i\alpha})}^s\overset{h}{\underset{k=0}{\sum}}c_{k,-j}(\zeta_\nu e^{i\alpha}+z)^k(\bar{z}+\bar{\zeta_\nu}
  e^{-i\alpha})^{k+j}=\overset{\min\{h+j,h+s\}}{\underset{p=s}{\sum}}\frac{nR^{2p}}{(p-s)!p!}{\left(\frac{\partial}{\partial{z}}\right)}^{p-s}{\left(\frac{\partial}{\partial{\bar{z}}}\right)}^pf(z).
\]
In this case $h+j=m-1, h-(m-1)=-j.$

2. Let $h>m-1$. Then
\[
f_{-j}(z)=\overset{m-1}{\underset{k=0}{\sum}}c_{k,-j}z^{k}\bar{z}^{k+j}.
\]
Now,
\[
\overset{n}{\underset{\nu=1}{\sum}}{(\zeta_\nu
e^{i\alpha})}^s\overset{m-1}{\underset{k=0}{\sum}}c_{k,-j}(\zeta_\nu
e^{i\alpha}+z)^k(\bar{z}+\bar{\zeta_\nu}
e^{-i\alpha})^{k+j}=\overset{\min\{m-1+j,m-1+s\}}{\underset{p=s}{\sum}}\frac{nR^{2p}}{(p-s)!p!}{\left(\frac{\partial}{\partial{z}}\right)}^{p-s}{\left(\frac{\partial}{\partial{\bar{z}}}\right)}^pf(z).
\]
Then $m-1+j\neq{m-1}$ and condition of the second case does not fit.

Similarly we analyze functions with positive indices.

1. Again, $h\leq{m-1}$,
$f_j(z)=\overset{h}{\underset{k=0}{\sum}}c_{k,j}z^{k+j}\bar{z}^{k}.$

Then
\[
\overset{n}{\underset{\nu=1}{\sum}}{(\zeta_\nu
e^{i\alpha})}^s\overset{h}{\underset{k=0}{\sum}}c_{k,j}(\zeta_\nu
e^{i\alpha}+z)^{k+j}(\bar{z}+\bar{\zeta_\nu}
e^{-i\alpha})^k=\overset{\min\{h,h+j+s\}}{\underset{p=s}{\sum}}\frac{nR^{2p}}{(p-s)!p!}{\left(\frac{\partial}{\partial{z}}\right)}^{p-s}{\left(\frac{\partial}{\partial{\bar{z}}}\right)}^pf(z).
\]
Now $h\neq{h+j}$ and condition of the first case does not fit.

2. Consider the case $h>m-1$.
\[
\overset{n}{\underset{\nu=1}{\sum}}{(\zeta_\nu
e^{i\alpha})}^s\overset{m-1}{\underset{k=0}{\sum}}c_{k,j}(\zeta_\nu
e^{i\alpha}+z)^{k+j}(\bar{z}+\bar{\zeta_\nu}
e^{-i\alpha})^k=\overset{\min\{m-1,m-1+j+s\}}{\underset{p=s}{\sum}}\frac{nR^{2p}}{(p-s)!p!}{\left(\frac{\partial}{\partial{z}}\right)}^{p-s}{\left(\frac{\partial}{\partial{\bar{z}}}\right)}^pf(z).
\]
Then we have $h-(m-1)=j$.

Combining the argument on the functions $f _j(z)$, we get
\[
f_{-j}(z)=\overset{h}{\underset{k=0}{\sum}}c_{k,-j}z^{k}\bar{z}^{k+j}
\]
and
\[
f_j(z)=\overset{m-1}{\underset{k=0}{\sum}}c_{k,j}z^{k+j}\bar{z}^{k}
.
\]
Now, considering the above and Lemma 2 and 4, we have
\[
f(z)=f_0(z)+f_+(z)+f_-(z),
\]
where
\[f_+(z)=\overset{h-(m-1)}{\underset{j=0}{\sum}}\overset{m-1}{\underset{k=0}{\sum}}c_{k,j}z^{k+j}\bar{z}^{k}=\overset{h}{\underset{l=0}{\sum}}\overset{m-1}{\underset{k=0}{\sum}}c_{k,l}z^l\bar{z}^{k}\;\;
(l=j+k),\]

 \[f_-(z)=\overset{m-1-h}{\underset{j=0}{\sum}}\overset{h}{\underset{k=0}{\sum}}c_{k,-j}z^{k}\bar{z}^{k+j}=\overset{m-1}{\underset{l=0}{\sum}}\overset{h}{\underset{k=0}{\sum}}c_{k,l}z^{k}\bar{z}^l\;\;(l=j+k).\]

So, the desired function $f(z)$ has the form (\ref{eq2}).
\end{proof}

\end{document}